\documentclass[12pt]{amsart}
\usepackage{amscd,amssymb,graphics}

\usepackage{amsfonts}
\usepackage{amsmath}
\usepackage{amsxtra}
\usepackage{latexsym}
\usepackage[mathcal]{eucal}

\usepackage{graphics,colortbl}

\usepackage[final]{pdfpages}

\input xy
\xyoption{all}
\usepackage{epsfig}
\usepackage[hidelinks]{hyperref} 

\newtheorem{thm}{Theorem}[section]
\newtheorem{cor}[thm]{Corollary}
\newtheorem{lem}[thm]{Lemma}
\newtheorem{prop}[thm]{Proposition}

\theoremstyle{definition}
\newtheorem{defn}[thm]{Definition}

\theoremstyle{remark}
\newtheorem{remark}[thm]{Remark}

\newtheorem{exa}[thm]{Example}
\newtheorem{exas}[thm]{Examples}

\numberwithin{equation}{section}

\numberwithin{equation}{section}

\newcommand{\delete}[1]{} 

\newcommand{\sig}{\sigma}

\def\al{\alpha}

\def\ga{\gamma}

\newcommand{\Ga}{\Gamma}
\newcommand{\del}{\delta}

\newcommand{\la}{\lambda}

\newcommand{\ka}{\kappa}

\newcommand{\graphh}{\rm{graph}}

\def\R {{\mathbb R}}
\def\N {{\mathbb N}}
\def\Z {{\mathbb Z}}
\def\Q {{\mathbb Q}}

\def\T {{\mathbb T}}

\def\Ocal{{\mathcal O}}
\def\Xcal{{\mathcal X}}
\def\Ucal{{\mathcal U}}

\def\GG{{\mathfrak{G}}}

\def\ol{\overline}

\newcommand{\br}{\vspace{4 mm}}

\newcommand{\id}{{\rm{id\,}}}

\def\Aut{{\mathrm Aut}\,}
\def\End{{\mathrm End}\,}

\def\Homeo{{\mathrm{Homeo}}\,}

\def\nor{\triangleleft}

\begin{document}

\title{On the virtual automorphism group of a minimal flow}

\author{Joseph Auslander}
\address{Mathematics Department, The University of Maryland,  College Park, MD 20742, USA}
\email{jna@math.umd.edu}

\author{Eli Glasner}
\address{Department of Mathematics, Tel Aviv University, Tel Aviv, Israel}
\email{glasner@math.tau.ac.il}



\date{June 9, 2019}

\begin{abstract}
We introduce the notions ``virtual automorphism group" of a minimal flow
and ``semi-regular flow" and investigate the relationship between the virtual
and actual group of automorphisms. 
\end{abstract}

\maketitle

\section*{Introduction}

The notion of a virtual subgroup was introduced in ergodic theory by W.A. Mackey
(see e.g. \cite{M}). Here is a brief description of the basic idea behind this
highly technical notion, as described by R. Zimmer \cite{Zi}.

\begin{quote}

If $X$ is an ergodic $G$-space, one of two mutually exclusive statements holds:

(i) There is an orbit whose complement is a nullset. In this case, $X$ is called essentially transitive.

(ii) Every orbit is a nullset. $X$ is then called properly ergodic.

In the first case, the action of $G$ on $X$ is essentially equivalent to the action defined by translation on $G/H$, 
where $H$ is a closed subgroup of $G$; 
furthermore, this action is determined up to equivalence by the conjugacy class of $H$ in $G$.
In the second case, no such simple description of the action is available, 
but it is often useful to think of the action as being defined by a ``virtual subgroup" of $G$. 
Many concepts defined for a subgroup $H$, can be expressed in terms of the action of $G$ on $G/H$;
 frequently, this leads to a natural extension of the concept to the case of an arbitrary virtual subgroup, 
 i.e., to the case of an ergodic $G$-action that is not necessarily essentially transitive. 
 Perhaps the most fundamental notions that can be extended in this way are those of a homomorphism, and the concomitant ideas of kernel and range. 
 These and other related matters are discussed in \cite{M}.
 
 \end{quote}

In some sense the concept of the ``Ellis group associated to a pointed minimal flow" 
(see Section \ref{sec-basic}) is an analogue of Mackey's virtual group in topological dynamics.
It became a keystone object in the abstract theory of topological dynamics, which
was developed by R. Ellis and collaborators in the 60's and 70's of the last century
(see e.g. \cite{E69}, \cite{EGS} and \cite{V}).

If one carries this idea a bit further, and one thinks of $A$, the Ellis group of a minimal flow $(X,T)$,
as a virtual subgroup, then the group $N_G(A)/A$, where $N_G(A)$ is the normalizer of $A$ in
the ambient group $G$ ( a subgroup of the enveloping semigroup of $(X,T)$), can be thought of as the ``virtual automorphism group" of the flow $(X,T)$.

In the present work we make this notion precise (Section \ref{sec-virtual}) 
and investigate the question of realization of the virtual automorphism group
as an actual group of flow automorphisms.

We thank, Ethan Akin and Andrew Zucker for their helpful comments.

\br

\section{Some notations and basic facts concerning minimal flows}\label{sec-basic}

In this work $T$ denotes an arbitrary (discrete) group. A {\em $T$-flow} 
$(X,T)$ on a compact Hausdorff space $X$ is given by a homomorphism
$\rho :T \to \Homeo(X)$ of $T$ into the group of self homeomorphisms of $X$.
We usually suppress the homomorphism $\rho$ from our notation of a flow (even when 
$\rho$ is not an injection) and we write $tx$ for the image of the point $x \in X$
under the homeomorphism $\rho(t)$ ($t \in T$).

In the next few paragraphs we will survey some of the basic definitions and facts
from the theory of abstract topological dynamics which will be repeatedly used in this work.
This  theory started with the classical monograph by Gottschalk and Hedlund \cite{GH} and was 
then greatly developed  by R. Ellis.
For more details we refer to the following monographs:
\cite{E69}, \cite{Gl76}, \cite{A88} and \cite{dV93}.

\br

The flow $(X,T)$ is {\em minimal} if every point in $X$ has a dense orbit. 
A pair of points $x, x' \in X$ is {\em proximal} if there is a net $t_i \in T$ and a point $z \in X$
such that $\lim t_i x = \lim t_i x' =z$.
We write $P[x]$ for the proximal cell of $x$ (i.e. the set of points proximal to $x$).
A point $x \in X$ is a {\em distal point} if it is proximal only to itself : $P[x] = \{x\}$.
A minimal flow is {\em point distal} if there is at least one distal point in $X$,
and it is {\em distal} if every point is distal.
Ellis has shown that in a metric minimal flow the existence of one distal point
implies that the set $X_0 \subseteq X$ of distal points is a dense $G_\del$ set.
A continuous map $\pi : (X,T) \to (Y,T)$ between two minimal flows is a {\em homomrphism}
(or an {\em extension}) 
if it intertwines the $T$-actions ($t\pi(x) = \pi(tx), \ \forall x \in X, t \in T$).
We say that the homomorphism is {\em proximal} if for every $y \in Y$ every pair of points in
$\pi^{-1}(y)$ is proximal, and that it is {\em distal}
if for every $y \in Y$  we have $P[x] \cap \pi^{-1}(y) = \{x\}, \ \forall x \in \pi^{-1}(y)$.

 \br

The {\em enveloping semigroup} of the flow $(X,T)$, denoted by $E(X,T)$,
is the closure of the set $\{\rho(t) : t \in T\}$ in the compact space $X^X$. 
This is indeed a compact subsemigroup of the semigroup (under composition of maps) $X^X$, and thus
for any fixed $p \in E(X,T)$,
right multiplication by $p$, $R_p : q \mapsto qp, \ (q \in E(X,T))$ is continuous on $E(X,T)$.
Left multiplication $L_p : q \mapsto pq,\ q \in E(X,T)$ is however often highly non-continuous
(usually not even measurable) unless $p$ is a continuous map. As the elements of $T$ are
continuous maps the homomorphism $t \mapsto L_t \ (t \in T)$ makes $E(X,T)$
a  $T$-flow.

It is well known that the semigroup $\beta T$, the Stone-\v{C}ech compactification of the discrete $T$,
is the universal point transitive $T$-flow and therefore also 
a universal enveloping semigroup. We will use this universality and often consider elements of $\beta T$ 
as maps in the enveloping semigroup of each and every $T$-flow under consideration.
The semigroup $\beta T$ admits many (for infinite $T$) minimal left ideals 
(which coincide with the minimal subflows). 
All these ideals are isomorphic to each other both as compact right topological semigroups and as minimal flows.
As usual, we will fix a minimal ideal  $M$ of $\beta T$. The universality of $\beta T$ 
implies that $(M,T)$ is a universal minimal flow. Ellis has shown that as a flow
$(M,T)$ is {\em coalescent}; i.e every endomorphism of $(M,T)$ is an automorphism, and thus
up to an automorphism $(M,T)$ is the unique universal minimal flow.
Each minimal ideal contains (usually many) idempotents and for convenience we usually fix one
such idempotent $u = u^2 \in M$. We denote the collection of idempotents in $M$ by the letter $J$.

It turns out that the set $G = uM \subset M$ is actually a group and moreover,
via the representation $g \mapsto R^{-1}_g, \ g \in G$, this group is isomorphic to 
the group $\Aut(M,T)$ of automorphisms of the flow $(M,T)$.
$M$ is the disjoint union of the collection of groups $\{vM : v \in J\}$ and 
each member $p$ of $M$ has a unique representation $p = vg$ where $v= v^2$ is an idempotent
in $M$ and $g$ is in $G$. We sometimes write $p^{-1}$ for $vg^{-1}$; this is
indeed the inverse element of $p$ in the group $vG$. 

If $(X,T)$ is minimal then for every $x \in X$ there is an idempotent $v \in M$ such that $vx =x$.
In other words $X = \bigcup \{vX : v \in J\}$. However, whereas $M = \bigcup \{vM : v \in J\}$
is a disjoint union, usually the sets $vX$ are not necessarily disjoint.
For example, a point $x $ in a minimal flow $(X,T)$ is distal iff $vx =x$ for every $v \in J$.
In particular a minimal flow $(X,T)$ is distal iff $X = uX = vX$ for all the idempotents $v \in J$.
Thus, in a minimal distal flow $E(X,T)$ is a group.

A minimal flow with a distinguished point $x_0 \in X$ is called a {\em pointed flow}
and we usually assume that $ux_0 = x_0$; i.e. $x_0 \in uX$.
We write 
$$
\mathfrak{G}(X,x_0) = \{g \in G : gx_0 = x_0\}.
$$
This subgroup of $G$ is called the {\em Ellis group} of the pointed flow $(X, x_0, T)$.
It is easy to check that for $g \in G$ we have
$$
\mathfrak{G}(X,gx_0) = g\mathfrak{G}(X,x_0)g^{-1}. 
$$

In the sequel we will often use the following fact. A homomorphism 
$\pi : (X,x_0,T) \to (Y,y_0,T)$ of pointed minimal flows is a proximal homomorphism iff
$\GG(X,x_0) = \GG(Y,y_0)$.

Note that when the group $T$ is abelian we have $tp = pt$ for every $t \in T$ and $p \in E(X,T)$.
In particular it follows that the subgroup $\{tu : t \in T\} \subset G$ is dense in $M$,
and for every minimal flow $(X,T)$ the subset $uX = G x_0 \cong G/A$ is dense in $X$.
More generally, following \cite{Gl75}, a minimal flow $(X,T)$ is called
{\em incontractible} when for every $k \in \N$ the set of almost periodic points is dense
in the product flow $(X^k, T)$ (equivalently when $u \circ uX = X$, see Section \ref{extensions} below).
The group $T$ is called {\em strongly amenable} \cite{Gl75}, when $(M,T)$ is incontractible
(equivalently, when $T$ admits no non-trivial minimal proximal flows).
The interpretation of the group $A =\GG(X,x_0)$ as a ``virtual group" is more meaningful for 
strongly amenable groups $T$,
as then, for every minimal flow $(X,T)$ the homogeneous space $uX = Gx_0 \cong G/A$
 ``approximates" the flow $X$.
(See \cite{FTVF} for a characterisation of strongly amenable groups.)
However, if we consider the relative theory (where one studies the collection 
of all the minimal extensions of a fixed minimal flow) the corresponding analogy 
will make sense for every acting group $T$.

\br

\section{The group of virtual automorphisms of a minimal flow}\label{sec-virtual}

Given a flow $(X,T)$ we let $\Aut(X,T)$ denote its group of automorphisms and $\End(X,T)$
its semigroup of endomorphisms. 
Given a subgroup $A < G$ we write
$N_G(A) = \{h \in G : h^{-1}A h = A\}$ and $N'_G(A) = \{h \in G : h^{-1}A h \subseteq A\}$

\begin{prop}\label{P1}
Let $(X,x_0,T)$ be a pointed minimal flow, $A = \mathfrak{G}(X,x_0)$, and let $\psi$ be an endomorphism 
of the flow $(X,T)$. Then there is an element $h \in N'_G(A)$ such that
$\psi = \phi_h$,
where for every $p \in M$,
$$
\phi_h(px_0) = phx_0.
$$
If $\psi$ is an automorphism then $h \in N_G(A)$ and $\psi^{-1} = \phi_{h^{-1}}$.
\end{prop}

\begin{proof}
Because $\psi$ is an endomorphism it commutes with $u$, hence $u\psi(x_0) 
= \psi(ux_0) = \psi(x_0) \in uX = Gx_0$,
and there exists an element $h \in G$ such that $\psi(x_0) = hx_0$.
Now for every $p \in \beta T$ (or $p \in E(X,T)$) we have
$$
\psi(px_0) = p \psi(x_0) = p hx_0.
$$
If $a \in A$ then
$$
hx_0 = \psi(x_0) = \psi(ax_0) = a\psi(x_0) = ahx_0,
$$
hence $h^{-1}ah \in A$,
so that $h \in N'_G(A)$.

If $\psi$ is an automorphism then 
$$
h\psi^{-1}(x_0) = \psi^{-1}(hx_0) = \psi^{-1}(\psi(x_0)) = x_0,
$$
hence $\psi^{-1}(x_0) = h^{-1} x_0$, so that, as above, $\psi^{-1} = \phi_{h^{-1}}$,
and also $hAh^{-1} \subseteq A$, hence  $h^{-1}Ah = A$.
\end{proof}

\begin{prop}\label{P2}
Let $(X,x_0,T)$ be a pointed minimal flow, $A = \mathfrak{G}(X,x_0)$ its Ellis group.
\begin{enumerate}
\item
Let  $h \in G$ be such that the map $\phi_h$,
where for every $p \in M$
$$
\phi_h(px_0) = phx_0,
$$
is well defined. Then $h \in N'_G(A)$ and 
$\phi_h$ is a continuous endomorphism; i.e. $\phi_h \in End(X,T)$.
\item
If $h \in G$ is such that both maps $\phi_h$ and $\phi_{h^{-1}}$ are well defined,
then $h \in N_G(A)$ and $\phi_h \in \Aut(X,T)$.
\end{enumerate}
\end{prop}

\begin{proof}
(1) We have for $a \in A$, $h x_0 = \phi_h(x_0) = \phi_h(ax_0) = ahx_0$, whence
$h^{-1}ah \in A$; i.e. $h \in N_G'(A)$.
Let $\pi : M \to X$ be the evaluation map $p \mapsto px_0, \ (p \in M)$. Let
$R_h : M \to M$ denote  right multiplication by $h$ and let
$$
L = (\pi \times \pi)({\graphh}(R_h)) =  (\pi \times \pi)\{(p, ph) : p \in M\} =
\{(px_0, phx_0) : p \in M\}.
$$
By our assumption $L$ is a graph of a map $\phi_h : X \to X, \ px_0 \mapsto phx_0, \ (p \in M)$.
Since the graph of $R_h$ is $T$-invariant ($R_h$ commutes with the elements of $T$),
we deduce that also $\phi_h$ commutes with the $T$-action.
Since $L$ is a closed subset of $M \times M$, it follows that  the map $\phi_h$ is continuous.
Thus $\phi_h \in \End(X,T)$.

(2) If $\phi_{h^{-1}}$ is also well defined then, as above, $hAh^{-1} \subseteq A$, hence $h \in N_G(A)$,
and we have $\phi_{h^{-1}} = (\phi_h)^{-1}$, so that $\phi_h \in \Aut(X,T)$.
\end{proof}

\begin{exa}
A minimal flow $(X,T)$ is said to be {\em coalescent} if
every endomorphism of $(X,T)$ is an automorphism.
This notion was introduced in \cite{A63}
where an example of a non-coalescent minimal $\Z$-flow was described.
This example, constructed as a factor of Ellis' ``two circles" flow, is not metrizable.
However a slight modification of the argument there will produce a
non-coalescent minimal subshift. 
Later, in \cite{D} Downarowicz constructed a Toeplitz flow $(X,T)$ which is not coalescent; i.e.
it admits an endomorphism which is not an automorphism.
These flows are both  almost one-to-one extension of their maximal equicontinuous factor 
$\pi : X \to Y$ (which is an irrational rotation in the modified Auslander example and an adding machine
in the Toeplitz flow).
This implies the following facts: (i) For any choice of a base point $x_0 \in uX$,
$\mathfrak{G}(X, x_0) = A    $ is a normal subgroup of $G$
(i.e. $N_G(A) = G$). (ii) The extension $\pi$ is proximal, whence
$\mathfrak{G}(X, x_0) = \mathfrak{G}(Y, y_0) = A$ (with $y_0 = \pi(x_0)$).
Now let $\phi$ be an endomorphism of $(X,T)$ which is not one-to-one.
As in Proposition \ref{P1} we have $\phi = \phi_h$ for some $h \in G$
and, although here $h^{-1}Ah =A$, yet $\phi_h$ is not an automorphism.
This phenomenon however can not occur for distal flows, as we can see in the following proposition.
\end{exa}

\begin{prop}\label{P3}
Let $(X,T)$ be a minimal distal flow. If $h \in N_G(A)$ is such that
the map $\phi_h(px_0) = phx_0$ is well defined then it is an automorphism.
\end{prop}

\begin{proof}
For distal $(X,x_0,T)$ we have $Gx_0 = X$. 
Suppose $gx_0 \not = g'x_0$ for some $g, g' \in G$ but $\phi_h(gx_0) = \phi_h(g'x_0)$. Then
$gx_0 \not = g'x_0$ implies that $b =g^{-1}g' \not\in A$, and
$ghx_0 = g'hx_0$, hence $x_0 = h^{-1}g^{-1}g' hx_0$, implies that $ h^{-1}b h \in A$. 
Thus $b \in hAh^{-1} \setminus A$, contradicting our assumption that $h \in N_G(A)$.
\end{proof}

We call  $N_G(A)/A$ the group of  {\em virtual automorphisms} of the minimal flow $(X,T)$.

\br

\section{Semi-regular flows}
The notion of regular minimal flows was introduced in \cite{A66}. 
A minimal flow $(X,T)$ is {\em regular} if for any pair of points $x, y \in X$ there 
is an automorphism $\psi \in \Aut(X,T)$ such that the pair $(x, \psi(y))$ is proximal;
iff  for every almost periodic point $(x, y) \in X^2$, there is an
automorphism $\psi$ of $X$ such that $y = \psi(x)$.
Equivalently we can say that $(X,T)$ is regular iff 
(for $x_0 \in uX = Gx_0$)
for every $g \in G$, there is an automorphism
$\psi \in \Aut(X,T)$ such that $\psi(x_0) = gx_0$
(recall that $\psi \circ u = u \circ \psi$ and that for every $x \in X$ the point $x$ is proximal to $ux$).

\begin{defn}
We say that a pointed minimal flow $(X,x_0,T)$ 
(with $\mathfrak{G}(X,x_0)=A$), is {\em semi-regular} (SR for short) if 
for every $h \in N_G(A)$ there is an automorphism $\psi \in \Aut(X,T)$ such that
$\psi(x_0)= hx_0$.
More generally, given a subgroup $\Ga < N_G(A)$, we say that 
$(X,x_0,T)$ is {\em $\Ga$-semi-regular} if 
for every $h \in \Ga$ there is an automorphism $\psi \in \Aut(X,T)$ such that
$\psi(x_0)= hx_0$. (See also \cite{H}.)
Clearly then a minimal flow $(X,T)$ is regular iff it is SR with $A \nor G$.
\end{defn}

Thus a minimal flow is SR iff every virtual automorphism of $(X,T)$ is realized, so that
$N_G(A)/A \cong \Aut(X,T)$. 

\begin{remark}
The dependence on the various choices we made in order to formulate the definition of the SR 
property; namely the choice of $M$ in $\beta T$, the choice of $u$ in $J$ and finally the choice
of $x_0$ in $uX$, are either immaterial or have the effect of replacing a subgroup of $G$
by some conjugate. 
\end{remark}

\begin{exas}
\begin{enumerate}
\item
Every regular flow is SR (clear).
\item
Every minimal distal flow is SR (see Section \ref{distal}).
\item
Every minimal proximal flow is SR (a minimal proximal flow is regular).
\item
If $(X,T)$ is SR then so is its maximal highly proximal extension (see Section \ref{extensions}).
\item
The Morse minimal set is SR (see Section \ref{Morse}).
\item 
The Sturmian minimal set is not SR (see Section \ref{distal}).
\item
Toeplitz flows are not SR (see Section \ref{distal}).
\item
More generally a metrizable almost automorphic flow which is not equicontinuous is never SR
(see Section \ref{distal}).
\item
A regular point distal minimal flow is distal (see Section \ref{distal}).
\item
Every minimal flow admits a proximal extension which is SR (see Section \ref{extensions}).
\end{enumerate}
\end{exas}

\begin{remark}
We note that an element $g \in G$ belongs to $N_G(A)$ iff $\GG(X, gx_0) =\GG(X,x_0) =A$.
For a commutative $T$ we have $\GG(X,tx_0) =A$ for every $t \in T$
(hence every $tu \in G$). Thus for commutative $T$ the set 
$\{gx_0 : g \in G \ \&\  \GG(X,gx_0) = A\} \supset Tx_0$, is always dense in $X$.
\end{remark}

One can easily relativize the notion of SR.

\begin{defn}\label{def-SRext}
Let $\pi : (X,x_0,T) \to (Y,y_0,T)$ be a homomorphism of pointed minimal flows.
Let $\GG(X,x_0) = A < \GG(Y,y_0) = F$ (where $y_0 = \pi(x_0)$).
We say that $\pi$ is an {\em SR extension} if for every $h \in N_F(A) =  F \cap  N_G(A)$
the map $\phi_h : X \to X$ defined by 
$\phi_h(px_0) = phx_0$ is a well defined automorphism of $(X,T)$.
\end{defn}

\section{Distal minimal flows are semi-regular}\label{distal}

\begin{thm}
Every minimal distal flow is SR.
\end{thm}

\begin{proof}
Let $(X,x_0,T)$ be a minimal distal flow with $A = \mathfrak{G}(X,x_0)$.
We have $X = Gx_0$ and for $h \in N_G(A)$ we 
let $\phi_h : X \to X$ be defined by
$$
\phi_h(px_0) = phx_0, \ (p \in M).
$$
To see that this is well defined suppose $px_0 = qx_0$ for some  $p, q \in M$.
We have to show that $phx_0 = qhx_0$.
Now in a distal flow the enveloping semigroup $E(X,T)$ is a group
and the the image of $M$ under the canonical map from $\beta T$ onto $E(X,T)$ is
surjective. Thus we can consider $p, q $ as elements of the group $G = E(X,T)$.
As $px_0 = qx_0$, it follows that $p^{-1}qx_0 = x_0$,
hence $a = p^{-1}q \in A$. Thus $q = pa$ and we get
$$
\phi_h(qx_0) = qhx_0 = pahx_0 = ph(h^{-1}ah)x_0 = phx_0 = \phi_h(px_0).
$$
The continuity of $\phi_h$ follows from the continuity of right multiplication
$R_h$ on $M$ and it follows that $\phi_h$ is an endomorphism of $(X,T)$.
Finally, as the same argument applies for $\phi_{h^{-1}}$, we
conclude that $\phi_h \in \Aut(X,T)$.
\end{proof}

\begin{exa}
In \cite{PW} the authors construct an example of a minimal metric
cascade $(X,T)$ which is not coalescent.
If $\phi$ is an endomorphism of $(X,T)$ which is not an automorphism,
then $\phi = \phi_h$ for some  $h \in G$ for which $h^{-1}Ah \subsetneq A$,
where $A = \GG(X,x_0)$ for some choice of $x_0 \in X$.
In fact, as we have seen above (Proposition \ref{P3}), in a minimal distal flow
$h^{-1}Ah = A$, would imply that $\phi = \phi_h$ is an automorphism.
\end{exa}

\begin{prop}
A minimal point distal regular flow is distal.
\end{prop}

\begin{proof}
Let $x_0$ be a distal point. Given $x \in X$ there is, by regularity, an automorphism $\psi \in \Aut(X)$
such that the points $\psi(x)$ and $x_0$ are proximal. But $x_0$ being a distal point, we have 
$\psi(x) = x_0$, and it follows that $x$ is also a distal point. Thus $X$ is a distal flow.

\end{proof}

\begin{cor}
For abelian $T$ a minimal metric almost automorphic flow which is SR is actually equicontinuous.
Thus Toeplitz flows and Sturmian like flows are not SR.
\end{cor}

\begin{proof}
By definition an almost automorphic flow is a metric minimal
flow $(X,T)$ such that the homomorphism $\pi : (X,T) \to (Y,T)$ 
from $(X,T)$ onto its maximal equicontinuous factor $(Y,T)$, is an
almost one-to-one extension.
Such a flow is point distal and it satisfies $\GG(X,x_0) = \GG(Y,y_0) = A \nor G$
(here we use the assumption that $T$ is abelian). 
Thus $N_G(A) = G$ and semi-regularity for $(X,T)$ is the same as regularity.
Now the previous proposition applies.
\end{proof}

It is natural to ask whether a minimal, point distal, flow which is not distal
can be SR. We will see in Section \ref{Morse} below that the Morse minimal
cascade (i.e a $\Z$-flow), which is point distal, metric and not distal, is in fact semi-regular.

\br

We also have an analogous statement concerning distal extensions.

\begin{prop}\label{distal-ext}
A distal extension $\pi : (X,x_0,T) \to (Y,y_0,T)$ of pointed minimal flows
is an SR extension.
\end{prop}

\begin{proof}
Let $\GG(X,x_0) = A < \GG(Y,x_0) = F$, and let $h \in N_F(A)$.
We have to show that the map $\phi_h$ is well defined.
Suppose then that $px_0 = qx_0$ for $p, q \in M$.
Then $a = up^{-1}q \in A$, hence
$h^{-1}ah = a' \in A$ and we have
\begin{equation}\label{prox}
phx_0 = pha'x_0 = ph(h^{-1}ah) x_0 =
pahx_0 = pp^{-1}qhx_0 = vqhx_0,
\end{equation}

where $v= v^2$ is the unique  idempotent in $J$ such that $vp =p$.
Thus the points $phx_0$ and $qhx_0$ are proximal.
On the other hand we have
$\pi(phx_0) = ph\pi(x_0) = phy_0 = py_0$ and
also $\pi(qhx_0) = qh\pi(x_0) = qhy_0 = qy_0$.
Since by assumption $px_0 = qx_0$ we also have $py_0 = qy_0$,
so that $phx_0$ and $qhx_0$ are in the same $\pi$ fiber.
Since $\pi$ is a distal extension we conclude that the points $phx_0$ and $qhx_0$ are both 
proximal and distal, whence equal.
Thus $\phi_h$ is a well defined element of $\End(X,T)$ and, as the same argument applies
to $\phi_{h^{-1}}$, we see that $\phi_h \in \Aut(X,T)$.
\end{proof}

A similar argument yields the following proposition.

\begin{prop}
Let  $\pi : (X,x_0,T) \to (Y,y_0,T)$ be a distal extension
with $Y$ being SR.
If, in addition,  we assume that $N_G(A) \subset N_G(F)$,
then $X$ is also SR.
In particular, this is the case when $Y$ is regular.
\end{prop}

\begin{proof}
Let $h \in N_G(A)$, we show that $\phi_h : X \to X$,
defined by $\phi_h(px_0) = phx_0, \ (p \in M)$ is well defined.
Assuming $px_0 = q x_0$ we have
\begin{gather*}
\pi(phx_0) = phy_0 = \tilde{\phi}_h(py_0)\\
\pi(qhx_0) = qhy_0 = \tilde{\phi}_h(qy_0),
\end{gather*}
where $\tilde{\phi}_h$ is the element of $\Aut(Y)$ defined by $h$
(recalling that $Y$ is SR and that $N_G(A) \subset N_G(F)$).
Since $px_0 = qx_0$ it follows that $py_0 = qy_0$,
whence also $\tilde{\phi}_h(py_0)=\tilde{\phi}_h(qy_0)$.
Thus the points $phx_0$ and $qhx_0$ lie in the same $\pi$-fiber, and therefore are distal.
As in the proof of Proposition \ref{distal-ext} (\ref{prox}) they are also proximal points hence equal.

The last assertion follows since by regularity $F \nor G$ and  $N_G(A) \subseteq N_G(F) = G$.
 \end{proof}

%
%

\br


\section{The Morse minimal set is semi-regular}\label{Morse}

We have already raised the question whether a minimal point distal, non-distal
flow can be SR. In this section we will show that the classical Morse minimal
set provides such an example. This example will also show that 
there is a  minimal metric point distal SR $\Z$-flow which is not equicontinuous. 
This should be contrasted
with the fact that a regular 
Proximal-Isometric (PI) $\Z$-flow
which is not equicontinuous is necessarily non-metrizabe (see \cite{Gl92}).
For a description and detailed analysis of the Morse minimal set we refer to \cite[Chapter 12]{GH}.

\begin{thm}\label{th-Morse}
The Morse minimal set is semi-regular.
\end{thm}

\begin{proof}
Let $(X,S)$ denote the Morse minimal flow. Here we deviate from our 
usual notation and use the letter $S$ to denote the shift homeomorphism
that generates the $\Z$-flow on $X$,  a subshift $X \subset \{0,1\}^\Z$.
We know that $(X,S)$ has the following structure:
$$
X \overset\sig\to Y \overset\theta\to Z,
$$
where 
(i) $\pi = \theta \circ \sig$ is the homomorphism of $(X,T)$ onto its
maximal equicontinuous factor (a dyadic adding machine), (ii)
$\theta$ is an almost one-to-one extension, and $\sig$ is a $\Z_2$ group extension.
More precisely, there is a point $z_1 \in Z$ such that $\theta^{-1}(z_1) = \{y_1, \bar{y}_1\}$
and for every point $z \in Z$ which is not in the orbit of $z_1$, we have that 
$\theta^{-1}(z)$ is a singleton. Finally on $X$ there is an involution $\kappa $
 (a self homeomorphism satisfying $\kappa^2 = \id$) which commutes with the shift.
 (It sends the sequence $x \in X \subset \{0,1\}^\Z$ into the ``fliped" sequence 
 $\kappa(x) = x'$, where $x'(n) = {x(n)'}$, and $0'=1,\ 1'=0$.)
 The map $\sig : X \to Y$ is then the quotient map under the action of $\Z_2 = \{\id, \kappa\}$.
 We let $\pi^{-1}(z_1) = \sig^{-1}\{y_1, \bar{y}_1\} =
 \{x_1, x'_1, \bar{x}_1, \bar{x}_1'\}$.
 
 Next fix a point $x_0 \in uX$ which is not on the orbits of
 the four points $ \{x_1, x'_1, \bar{x}_1, \bar{x}_1'\}$.
 We let $y_0 = \sig(x_0)$ and $z_0 = \theta(y_0)$. 
 Also set  $ \GG(X,x_0)  = A$ and $\GG(Y,y_0) = \GG(Z,z_0) = F$.
 We then have $F \nor G$ and $Z \cong G/F$, and
 $A \nor F$ and $\Z_2 \cong F/A$.
 It is also shown in \cite{GH} that the group
 $\Aut(X,S)$ is the group 
 $\{S^n : n \in \Z\} \oplus \{\id, \kappa\} \cong\Z \oplus \Z_2$.

 Since $\theta^{-1}(z_0) =\{y_0\}$ it follows that 
 $\pi^{-1}(z_0) = \sig^{-1}(y_0) = \{x_0, x'_0\}$ and both $x_0$ and
 $x'_0$ are distal points. 
 
  
 Let now $x$ be an arbitrary point in $uX$ and set
 $W =  \ol{\Ocal(x_0,x)}$, the orbit closure of the point $(x_0,x) \in X \times X$.
 Because $x \in uX$ the point $(x_0,x)$ is an almost periodic point of 
 the product flow $X \times X$, so that $W$ is a minimal flow.
 Also, there exists $g \in G$ with $x = gx_0$.
 Let $y = \sig(x)$ and $z = \pi(x) = \theta(y)$.
 
 We claim that for every $(a,b) \in W$ we have
 $$
 W[a] : = \{c \in X : (a,c) \in W\} \subseteq \{a\} \times \pi^{-1}(\pi(b)).
 $$
 
 To see this note that if $(x_0,c) \in W$ then $(x_0,c) = p(x_0,x)$ for some
 $p \in M$ and  
 $$ 
 (\pi \times \pi)(x_0,c) =(\pi \times \pi)(p(x_0,x))= (pz_0, p\pi(x)) = 
 (z_0, pz).
$$
 Since $Z$ is equicontinuous, $pz_0 = z_0$ implies that 
 $pz =z$, so that $\pi(c) =z$. Thus
 $$
 W[x_0]  = \{c \in X : (x_0,c) \in W\} \subseteq \{x_0\} \times \pi^{-1}(z).
 $$
 and therefore
 $W[px_0] \subseteq \pi^{-1}(\pi(px))$ for every $p \in M$, as claimed.
 
 \br

We now consider two possible cases:

\br

{\bf Case 1:}
Suppose that $(x_0,x') = (x_0, \ka(x)) \in W$.
(Note that this implies that $(\id \times \ka)(W) \cap W \not=\emptyset$, whence
$(\id \times \ka)(W) = W$.)
Since $(x_0, x') \in uW = G(x_0,x)$, there exists $h \in G$ such that
$h(x_0,x) = (x_0,x')$. In particular $hx_0 = x_0$, whence $h \in A$.
On the other hand 
$x' = \ka(x) = \ka(gx_0) = gx'_0 =  hx = hgx_0$, 
hence $x'_0 = g^{-1}hg x_0$, hence $g^{-1}hg \not \in A$;
i.e. $g \not\in N_G(A)$.

\br

{\bf Case 2:}
Now assume that $(x_0,x') \not \in W$.
As in Case 1, we deduce that
\begin{equation}\label{kappa}
{\text {if $(a,b) \in W$ then $(a, \ka(b)) \not\in W$}}.
\end{equation}

Let $P$ denote the projection map from $W$ onto $X$ (as its first coordinate).
We have $P^{-1}(a) = W[a] \subseteq \{a\} \times \sig^{-1}(\theta^{-1}(z))$.

%
%
%
%
%
%
%
%
%
%

\br

{\bf Claim:} For every $a \in X$ we have $P^{-1}(a) = \{(a,\phi(a))\}$
for a surjection $\phi : X \to X$.

\br

We recall here that 
the enveloping semigroup $E(X,S)$ of the Morse flow has exactly
two minimal ideals, say $I_1, I_2$, each containing exactly two idempotents,
say $J_1 = \{u, v\} \subset I_1$ and $J_2 = \{\tilde{u}, \tilde{v}\} \subset I_2$
(see \cite{HJ} and \cite{S}). 
Now for the point $a$ we have at least one of the possibilities $ua=a$ or $va=a$.
We will assume that $ua =a$ and the other case is treated in the same way
with $v$ replacing our usual $u$.

Fix some $b \in X$ with $(a,b) \in W$ and $ub=b$.
We write $\eta = \sig(b)$ and $\zeta= \pi(b) = \theta(\eta)$.

To prove our claim we again consider two cases:

\br

{\bf Case $2_a$:} The point $\eta$ is a non-split point; i.e.
$\theta^{-1}(\zeta) =\{\eta\}$.

In this case we have
$P^{-1}(a) \subseteq \{a\} \times \sig^{-1}(\eta)
= \{a\} \times \{b, b'\}$, and, since by assumption
$(a, b') \not\in W$, we indeed have
$P^{-1}(a) = \{(a,b)\}$, as claimed (putting $b = \phi(a)$).

\br

{\bf Case $2_b$:} The point $\eta$ is a split point; i.e.
$\theta^{-1}(\zeta) =\{\eta, \bar{\eta}\}$, with $\eta \not = \bar{\eta}$.

In this case we have $\pi^{-1}(\zeta) = \{b, b', \bar{b}, \bar{b}'\}$.
Suppose $(a, \bar{b}) \in W$ (the case  $(a, \bar{b}') \in W$ is symmetric).
By the general theory of enveloping semigroups there is in $I_2$  
a unique idempotent $\tilde{u}$ equivalent to $u$;
i.e. $u\tilde{u} = \tilde{u}$ and $\tilde{u}u= u$ (see e.g. 
\cite[I, Proposition 2.5]{Gl76}).
Both $u$ and $\tilde{u}$ act as the identity on $uX$, however, as they are
distinct elements of $E(X,S)$,
we must have
\begin{gather}\label{collaps}
\begin{split}
u(b,b',\bar{b}, \bar{b}') = (b,b',b,b')\\
\tilde{u}(b,b',\bar{b}, \bar{b}') = (b,b',b',b)
\end{split}
\end{gather}
(or vice versa).
(In fact, if $u\bar{b}= \tilde{u}\bar{b}$ then also $uS^n\bar{b}= \tilde{u}S^n\bar{b}$ and
$uS^n\bar{b}'= \tilde{u}S^n\bar{b}'$
for every $n \in \Z$, whence $u = \tilde{u}$, which is impossible.)
Thus we have $\tilde{u}(a, \bar{b}) = (a,b') \in W$
(or $u(a, \bar{b}) = (a,b') \in W$) contradicting our assumption that 
$(a,b') \not \in W$. 
We have shown, in view of (\ref{kappa}), that indeed $P^{-1}(a) = \{(a,b)\}$,
and we let $\phi(a) =b$.

Given $c \in X$, there is a point $a \in X$ such that $(a,c) \in W$
and as also $(a, \phi(a)) \in W$, we have $\phi(a) = c$.
This shows that $\phi$ is surjective and our claim is proven.

\br

We now have $W = \{(a, \phi(a)) : a \in X\}$ and it follows that $\phi \in \End(X,S)$.
Since $g^{-1}(x_0,x) = g^{-1}(x_0,gx_0 ) = (g^{-1}x_0, x_0) \in W$
we conclude, by symmetry, that $\phi \in \Aut(X)$.

\br

To sum up, we have shown that for every $h \in G$ either $h \not \in N_G(A)$ or
$\phi_h \in \Aut(X,S)$; in other words, we have shown that $(X,S)$ is semi-regular.
Thus the group of virtual automorphisms $N_G(A)/A$ is realized as 
$\Aut(X,S) = \{S^n : n \in \Z\} \oplus \{\id, \ka\} \cong \Z \oplus \Z_2$.
\end{proof}

\begin{remark}
Using \cite{GH}'s notation we can write
$ \{x_1, x'_1, \bar{x}_1, \bar{x}_1'\}
 = \{\mu, \mu', \nu,\nu'\}$, where 
 \begin{gather*}
 \mu = \breve{Q}Q, \qquad
 \mu' = \breve{Q}'Q', \\
 \nu= \breve{Q'}Q, \qquad
 \nu' = \breve{Q}'Q'. 
 \end{gather*}
From this description it follows immediately that 
the pairs $\{\nu, \mu\} = \{\bar{x}_1, x_1\}$ and $\{\nu, \mu'\} = \{\bar{x}_1, x'_1\}$ are 
positively asymptotic and negatively asymptotic pairs, respectively, hence proximal. This directly implies (\ref{collaps}).
In fact, this argument can be used to prove that indeed $E(X,T)$ has exactly two minimal ideals, 
each having exactly two idempotents.
\end{remark}

\begin{remark}
We note that the proof of  Theorem \ref{th-Morse} shows also that the Morse flow is coalescent.
\end{remark}

\br

\section{Every minimal flow admits a proximal extension which is SR}\label{extensions}

\begin{thm}
Let $(X,x_0,T)$ be a minimal flow with $\GG(X,x_0)= A$, and let $\Ga \leq N_G(A)$ be a subgroup.
Then there exists a minimal pointed flow
$Z_\Ga = (Z, z_0,T)$ and a homomorphism $\pi : Z \to X$ such that:
\begin{enumerate}
\item
$\GG(Z,z_0) =A$ (so that $\pi$ is a proximal extension).
\item
 $\Ga A/A  \leq  \Aut(Z)$.
 \item
 For every minimal flow $(Y,y_0,T)$ which is a proximal extension $\eta :(Y, y_0,T) \to (X,x_0,T)$
such that  $\Ga A/A < \Aut(Y)$, there is a commutative diagram:
 
 \begin{equation*}
\xymatrix
{
(Y,y_0) \ar[d]_\eta \ar[r]^\la & (Z,z_0) \ar[dl]^\pi\\
(X,x_0) &
}
\end{equation*} 
\end{enumerate}
\end{thm}

\begin{proof}
Let $C = \{hx_0 : h \in \Ga\} \subset X$.
Let $z_0 \in X^C$ be the point
$$
z_0(hx_0) = hx_0, \ (h \in \Ga).
$$
Let $Z = \ol{\Ocal_T(z_0)} \subset X^C$
(alternatively $Z = \bigvee_{h \in \Ga} (X, hx_0)$).
Because $C \subseteq uX$, it follows that the point
$z_0$ is an almost periodic point of the product flow
$X^C$, hence $Z$ is a minimal flow.
The projection on the $x_0 \in C$ coordinate is
a homomorphism $\pi : (Z,z_0,T) \to (X, x_0,T)$.

Clearly $az_0 = z_0$ for every $a \in A$, so that $A \subseteq \GG(Z,z_0)$. 
Conversely, if $bz_0 = z_0$ for some $b \in G$ then
$$
(bz_0)(x_0) = bz_0(x_0) = bx_0 = z_0(x_0) = x_0,
$$
hence $b \in A$, $\GG(Z,z_0) \subseteq A$, hence $\GG(Z,z_0) = A$.

For $h \in \Ga$ the map
$$
\phi_h : Z \to Z, \quad \phi_h(pz_0) = phz_0, \ (p \in M),
$$
is well defined.
In fact if $pz_0 = qz_0$ for $p, q \in M$, then for all $h' \in \Ga$
\begin{gather*}
(phz_0)(h'x_0) = ph(z_0(h'x_0))  = phh'x_0 = pz_0(hh'x_0)\\
(qhz_0)(h'x_0) = qh(z_0(h'x_0))  = qhh'x_0 =  qz_0(hh'x_0).
\end{gather*}
As $\Ga$ is a group, $hh' \in \Ga$ and $hh'x_0 \in C$. 
Therefore, by our assumption,  $phz_0(hh'x_0)  = qhz_0(hh'x_0)$, whence
 $phz_0  = qhz_0$.
We conclude, as in Section \ref{sec-virtual}, that $\phi_h \in \Aut(Z)$. 
We thus get $\Ga A /A  \leq  \Aut(Z)$.

Next assume that $(Y,y_0,T)$ is as in (3).
By assumption $\Ga A /A  \leq  \Aut(Y)$ and it follows that for each $h \in \Ga$
the function $\phi_h(py_0) = phy_0$ is a well defined element of $\Aut(Y,T)$.
Thus the map $\eta \circ \phi_h : Y \to X$ satisfies
$$
(\eta \circ \phi_h)(y_0) = \eta(hy_0) = h x_0,
$$
and $\eta \circ \phi_h : (Y, y_0, T) \to (X, hx_0, T)$ is a homomorphism of pointed flows.
It then follows that there is a homomorphism
$$
\la : (Y, y_0, T) \to \bigvee_{h \in \Ga}(X, hx_0,T) \cong Z_\Ga
$$
which satisfies $\eta = \pi \circ \la$.
\end{proof}

\begin{remark}
Note that when $(X,T)$ is metrizable and the group $\Ga A/A$ is countable
then the SR flow $Z_\Ga$ is also metrizable. In particular this is the case for
$\Ga = \langle \gamma_0 \rangle = \{\ga_0^n : n \in \Z\}$ for some $\ga_0 \in N_G(A)$.
\end{remark}

\begin{cor}
Every minimal flow admits a proximal extension which is SR
\end{cor}

\begin{proof}
Take $\Ga = N_G(A)$.
\end{proof}

\begin{defn}
We write $X_{SR}$ for the flow $Z_{N_G(A)}$. With this notation
it is easy to check that $(X,T)$ is SR iff $X = X_{SR}$.
\end{defn}

\br

\begin{exa}
Let $(Z, R_\al)$ denote the rotation by $\al \in \R \setminus \Q$ on the circle
$Z = \R/\Z$, \ $R_\al(x) = z + \al \pmod 1$.
Let $(X, S)$ be the Sturmian flow with $\pi : X \to Z$ being its maximal equicontinuous factor.
Then the flow $Z$ is regular and the extension $\pi$ is almost one-to-one (so that $X$ is almost automorphic).
The regularizer of $(X,S)$, which is the same as $X_{SR}$, is the Ellis' two circles flow $\tilde{X}$, 
a nonmetrizable flow (see e.g. \cite[Example 14.10]{GMe} for more details on this).
We have $\GG(Z,z_0) = \GG(X,x_0)=\GG(\tilde{X}, \tilde{x}_0)=A$, and 
we observe that the virtual automorphism group
$N_G(A)/A$ is realized on $Z$ as the compact group $\Aut(Z) \cong \T = \R/\Z$, on $\tilde{X}$
again as $\Aut(\tilde{X}) \cong \R/\Z$, but with the discrete topology, and it is mostly nonrealizable on $X$,
where $\Aut(X) = \{S^n : n \in \Z\}$.
\end{exa}

\br

Given any minimal flow $(X,T)$, we will next describe
another natural construction that yields an SR flow $X^{SR}$ which is a proximal extension of $X$.
In the collection of all the minimal flows which are SR and are proximal extensions of $X$,
the flow $X_{SR}$ is the minimum and the flow $X^{SR}$ is the maximum (with respect to being a factor).

Let $(X,x_0,T)$ be a minimal flow with $\GG(X,x_0)= A$. Let
$X^{SR} = \Pi(A)$ denote the minimal flow which is the maximal proximal extension of $X$.
As described in \cite{Gl76} this flow can be presented as a quasifactor of $M$, as follows:
$$
\Pi(A) = \{ p \circ A : p \in M\}.
$$
Moreover the map $\pi : M \to \Pi(A)$ is such that
$\pi^{-1}(p \circ (u\circ A)) = p \circ A, \ \forall p \in M$.
In particular the elements of $\Pi(A)$ form a partition of $M$.
(For more details on the quasifactor $\Pi(A)$ and the the circle operation see \cite[Chapter IX]{Gl76}.)

\begin{lem}\label{circ}
For $h \in N_G(A)$ and $p = \lim t_i \in M$,
$$
ph \circ A = p\circ hA = (p \circ A)h.
$$
\end{lem}

\begin{proof}
We have
\begin{gather*}
ph \circ A = p\circ h \circ A \supseteq p \circ hA \\
=\{\lim t_i h a_i : a_i \in A\} = \{\lim t_i  a'_i h : a'_i \in A \} \\
=p \circ Ah = (p \circ A)h  \in \Pi(A).
\end{gather*}
However, as the elements of $\Pi(A)$ form a partition of $M$,
we get $ph \circ A = p\circ hA = (p \circ A)h$.
\end{proof}

\begin{prop}\label{Pi}
For every pointed minimal flow $(X,x_0,T)$ with $\GG(X,x_0)= A$,
the universal minimal proximal extension $\Pi(A)$ of $X$ (which depends only on $A$) is semi-regular.
\end{prop}

\begin{proof}
Given $h \in N_G(A)$ and $p \in M$ we set $\phi_h(p \circ A) =ph \circ A$.
This is well defined since if $p \circ A = q \circ A$ then, by Lemma \ref{circ},
\begin{gather*}
ph \circ A = p \circ hA = p \circ Ah = (p \circ A)h \\
= (q \circ A)h = q \circ Ah = q \circ hA = qh \circ A. 
\end{gather*}
\end{proof}

\br

\begin{remark}
Since every $\tau$-closed subgroup $A$ of $G$ (see e.g. \cite[Chapter IX]{Gl76})
is the Ellis group of a minimal flow, namely $A = \GG(\Pi(A), u \circ A)$,
we get, in view of Proposition \ref{Pi}, that $N_G(A)/A \cong \Aut(\Pi(A))$,
hence we conclude that for every $\tau$-closed subgroup $A < G$, the 
group $N_G(A)/A$ is realized as 
an actual automorphism group of some minimal flow.
\end{remark}

We end this section with the following proposition (for information on the 
maximal highly proximal extension of a minimal flow see \cite{AG}) :

\begin{prop}
If $(X,T)$ is SR then so is its maximal highly proximal extension.
\end{prop}

\begin{proof}
One way to describe the maximal highly proximal extension $X^*$ of $X$
is as the qusifactor of $M$ obtained from a homomorphism $\pi : M \to X$, as follows:
$$
X^* = \{p \circ \pi ^{-1}(x_0) : p \in M\},
$$
where $x_0 \in uX$ and $\pi(p) = px_0, \ (p \in M)$.
As with the quasifactor $\Pi(A)$, it can be shown that 
 $\{p \circ \pi ^{-1}(x_0) : p \in M\}$ is a partition of $M$.
As $X$ is SR we have $N_G(A)/A \cong \Aut(X)$ and it suffices to show
that every $\psi \in \Aut(X)$ lifts to an automorphism $\psi^*$ of $X^*$.
We define $\psi^*(p \circ \pi^{-1}(x_0)) = p \circ \pi^{-1}(\psi(x_0)), \ (p \in M)$.
If $p \circ \pi^{-1}(x_0) = q \circ \pi^{-1}(x_0)$, for $p, q \in M$, then 
also $px_0 = qx_0$, whence
$$
\psi(px_0) = \psi(qx_0) \in p \circ \pi^{-1}(\psi x_0) \cap q \circ \pi^{-1}(\psi x_0),
$$
hence $p \circ \pi^{-1}(\psi x_0) = q \circ \pi^{-1}(\psi x_0)$.
(It is not hard to check that $p \circ \pi^{-1}(x) \in X^*$ for every $p \in M$ and every $x \in X$.)

%

\br

An alternative proof is as follows:

\br 

If $(X,T)$ is a minimal flow then, as a topological space, the maximal HPI extension of 
$X$, say $X^*$, is the Stone space of the Boolean algebra of regular open sets in $X$
(see e.g. \cite{Zu}).
It thus follows that every self-homeomorphism of $X$ lifts to $X^*$,
and clearly an automorphism of $(X,T)$ lifts to an automorphism of $(X^*,T)$.
\end{proof}

\br

\section{A Koopman representation of the virtual  automorphism group}

When a minimal flow $(X,T)$ is strictly ergodic 
(i.e. it admits a $T$-invariant probability measure and this measure is unique) 
then the group $\Aut(X,T)$ also preserves this measure.
In fact, if $\mu$ is the $T$-invariant measure on $X$ and
$\psi \in \Aut(X,T)$, then clearly the pushforward measure $\psi_*(\mu)$ is $T$-invariant as well,
and thus $\psi_*(\mu) =\mu$ by uniqueness.

The {\em Koopman representation} of $T$ associated to the measure 
preserving system $(X,\Xcal,\mu,T)$ is the representation on the Hilbert space $L_2(\mu)$ 
given by $t \mapsto U_t$, where for $t \in T$ the unitary operator $U_t$
is defined by $U_t(f) = f \circ t^{-1}$.

Now in a special case we are able to show that also the virtual automorphism group $N_G(A)/A$
admits such a faithful representation.

\begin{thm}
Let $(X,T)$ be a minimal, metrizable, point distal, uniquely ergodic flow
with $T$-invariant probability measure $\mu$, and suppose that $X$ is measure-regular,
where the latter property means that the 
$T$-invariant $G_\del$ set $X_0 \subseteq X$ consisting of distal
points, has measure one. Then each element of the virtual group automorphisms of the flow
defines an automorphism of the measure space $(X, \Xcal, \mu)$ and this correspondence
defines a unitary representation of the virtual group of automorphisms as a group
of unitary operators on the separable Hilbert space $L_2(\mu)$.
\end{thm}

\begin{proof}
As usual we pick a point $x_0 \in X_0$, so necessarily $ux_0 =x_0$, then let
$A = \GG(X,x_0)$. Our virtual automorphism group is the the group $N_G(A)/A$.
Clearly $uX = Gx_0 \supseteq X_0$. 
Next define, for $h \in N_G(A)$,  $\phi_h : uX \to uX$ by
$$
\phi_h(gx_0) = ghx_0, \ (g \in G)
$$
 Then the map $\phi_h$ is well defined and it is a homeomorphism of the 
 (usually not even measurable) set $uX$.
 In fact, if $gx_0 = g'x_0$ then $g^{-1}g'x_0 = x_0$, hence $a = g^{-1}g' \in A$ and
 $$
 g'hx_0 = g(g^{-1}g')h x_0 = gah x_0 = g h (h^{-1}ah) x_0 = ghx_0.
 $$
 The continuity of $\phi_h$ follows from the continuity 
 of right multiplication on $G$.
 Moreover, we have $t \phi_h = \phi_h t, \ \forall t \in T$.


 Now $\phi_h : X_0 \to  \phi_h(X_0)$ is a homeomorphism
 and it pushes the measure $\mu$ on $X_0$ to a measure $(\phi_h)_*(\mu)$ on $\phi_h(X_0)$.
 By uniqueness $(\phi_h)_*(\mu) = \mu$.
 
 In particular $\mu(X_0 \cap \phi_h(X_0)) =1$.
 Similarly $\mu(X_0 \cap \phi_h^n(X_0)) =1$ for every $n \in \Z$ and we conclude that the $T$-invariant
 dense $G_\del$ set $X_\infty = \bigcap_{n \in \Z} \phi_h^n(X_0)$ has measure $1$.
 Since it is also $\phi_h$-invariant this shows that $\phi_h$ is an automorphism of the measure space 
$(X, \Xcal, \mu)$.

Now the composition map $U_h : f \mapsto f \circ \phi^{-1}_h$ defines a unitary operator on
$L_2(\mu)$ and the map $h \mapsto U_h$ from $N_G(A)/A \to \Ucal(L_2(\mu))$ is
the desired unitary representation.
\end{proof}

\br

\section{Questions}

We conclude with the following list of questions:

\begin{enumerate}
\item
Given a $\tau$-closed subgroup $A$ of $G$ (see e.g. \cite[Chapter IX]{Gl76}), 
when is there a minimal {\bf metric} pointed
flow $(X, x_0, T)$ with $\GG(X,x_0)= A$ ?
\item
Given a minimal metric pointed flow $(X, x_0, T)$ when is there a metric SR proximal extension
of $X$ ?
\item (A. Zucker) 
For a minimal metric flow $(X,T)$ is there always a metrizable extension which is coalescent ?
(Of course $(M,T)$ is a non-metrizable one, when $T$ is infinite.)
\item
When is there a unitary representation of $N_G(A)/A$ on a separable Hilbert space ?
\item
In \cite{GTWZ} the authors show that for any infinite countable group $T$:
(i) There is a minimal metric flow $(X,T)$ such that $\Aut(X,T)$ 
 embeds every compact metrizable group, and 
 (ii) The group $\Aut(M(T), T) \cong G$ has cardinality $2^\mathfrak{c}$, the largest possible one.  
Is there an algebraic obstruction on an abstract group $\Ga$ of cardinality $\leq 2^\mathfrak{c}$ to be isomorphic 
to a (virtual) automorphism group of a minimal $T$-flow ? 
(Regarding this question see the work \cite{CP} and remark 11.6 in \cite{GTWZ}.)
\end{enumerate}
\br

\bibliographystyle{amsalpha}

\end{document}